\documentclass[a4paper]{amsart}
\usepackage{nicefrac}
\usepackage{array}
\usepackage{amsfonts,amsmath,amssymb,amsthm,amscd,amsxtra}
\usepackage{enumerate,verbatim}
\usepackage{rotating}
\usepackage{amsfonts,amsmath,amssymb,amsthm,amscd,amsxtra}
\usepackage{enumerate,verbatim}
\usepackage[all,2cell,ps]{xy}
\usepackage[notcite,notref]{}
\usepackage[pagebackref]{hyperref}
\usepackage{mathptmx}
\usepackage[notcite,notref]{}
\usepackage{amssymb,amstext,amsmath,amscd,amsthm,amsfonts,enumerate,graphicx,latexsym}
\usepackage[usenames]{color}
\makeatletter
\newcommand\@notni[2]{\mathrel{\rotatebox[y=#1]{180}{$#2\notin$}}}
\newcommand\notni{
\mathchoice
  {\@notni{0.57ex}\displaystyle}
  {\@notni{0.57ex}\textstyle}
  {\@notni{0.39ex}\scriptstyle}
  {\@notni{0.26ex}\scriptscriptstyle}
}
\makeatother

\usepackage{comment}
\newtheorem{thm}{Theorem}[section]
\newtheorem{cor}[thm]{Corollary}
\newtheorem{lem}[thm]{Lemma}
\newtheorem{prop}[thm]{Proposition}
\theoremstyle{definition}
\newtheorem{dfn}[thm]{Definition}
\newtheorem{rmk}[thm]{Remark}
\newtheorem{ques}[thm]{Question}
\newtheorem{conj}[thm]{Conjecture}
\newtheorem{ex}[thm]{Example}

\newtheorem*{claim*}{Claim}
\theoremstyle{remark}

\numberwithin{equation}{thm}

\def\depth{\operatorname{depth}}

\def\Ext{\operatorname{Ext}}

\def\Hom{\operatorname{Hom}}

\def\im{\operatorname{im}}
\def\ker{\operatorname{ker}}
\def\ll{\ell\ell}

\def\m{\mathfrak{m}}

\def\n{\mathfrak{n}}

\def\p{\mathfrak{p}}
\def\a{\mathfrak{a}}
\def\pd{\operatorname{pd}}

\def\soc{\operatorname{soc}}

\def\syz{\Omega}

\def\Tor{\operatorname{Tor}}
\def\Tr{\operatorname{Tr}}

\begin{document}

\title[On a class of Burch ideals and a conjecture of Huneke and Wiegand]{On a class of Burch ideals and a \\conjecture of Huneke and Wiegand }


\author[O. Celikbas]{Olgur Celikbas}
\address{Olgur Celikbas\\
Department of Mathematics \\
West Virginia University\\
Morgantown, WV 26506-6310, U.S.A}
\email{olgur.celikbas@math.wvu.edu}

\author[T. Kobayashi]{Toshinori Kobayashi} 
\address{Toshinori Kobayashi\\Graduate School of Mathematics, Nagoya University, Furocho, Chikusaku, Nagoya, Aichi 464- 8602, Japan
}
\email{m16021z@math.nagoya-u.ac.jp}

\thanks{Kobayashi's research was supported by JSPS Grant-in-Aid for JSPS Fellows 18J20660}
\maketitle{}

\begin{abstract} In this paper we are concerned with a long-standing conjecture of Huneke and Wiegand. We introduce a new class of ideals and prove that each ideal from such class satisfies the  conclusion of the conjecture in question. We also study the relation between the class of Burch ideals and that of the ideals we define, and construct several examples that corroborate our results.
\end{abstract} 

\setlength{\baselineskip}{15pt}
\section{Introduction}

Throughout $R$ denotes commutative, Noetherian local ring with unique maximal ideal $\m$ and residue field $k$. Moreover, all $R$-modules are assumed to be finitely generated. 

The present paper addresses the ideal case of a long-standing conjecture of Huneke and Wiegand; see \cite[4.6 and the discussion following the proof of 5.2]{HW1}. The version of the conjecture we are concerned with can be stated as follows:

\begin{conj} (Huneke and Wiegand \cite{HW1}) \label{HWC} If $R$ is a one-dimensional local domain and $I$ is a non-principal ideal of $R$, then $I\otimes_RI^{\ast}$ has (nonzero) torsion, where  $I^{\ast}=\Hom_R(I,R)$.
\end{conj}

The motivation of Conjecture \ref{HWC} comes from a celebrated conjecture of Auslander and Reiten \cite{AuBr}, as well as from the pioneering work of Huneke and Wiegand \cite{HW1}. There are several partial affirmative results concerning Conjecture \ref{HWC} in the literature; for example, it is remarkable that Conjecture \ref{HWC} holds over hypersurface rings \cite[3.1]{HW1} and over Cohen-Macaulay local rings of minimal multiplicity \cite[3.6]{HWS}. However, there is not much known about the conjecture, in general. In fact, Conjecture \ref{HWC} remains open, even if $R$ is Gorenstein and $I$ is minimally generated by two elements. An affirmative answer in this direction was initially established in \cite{Kurt}: Conjecture \ref{HWC} holds if $R$ is Gorenstein, and both $I$ and $I^{\ast}$ are minimally generated by two elements; this fact implies that Conjecture \ref{HWC} is true over Gorenstein domains of multiplicity at most five; see also \cite[1.5]{Taniguchi} and  \cite[3.1]{HWS}.

The aim of this paper is to make progress on Conjecture \ref{HWC}: we introduce a new class of ideals over local rings, and prove that each ideal from such class satisfies the torsion conclusion of Conjecture \ref{HWC}; see  Theorems \ref{thmnew} and \ref{t3}, and also Corollary \ref{cor1}. A consequence of our main result can be stated as follows; see Corollary \ref{cor2}.

\begin{thm} \label{thmintro} Let $(R, \m)$ be a one-dimensional local domain which is not regular, and let $I$ and $J$ be ideals of $R$. Assume $0\neq I\subseteq \m J$ and $(I:_RJ)=(\m I:_R \m J)$. Then $I\otimes_RI^{\ast}$ has (nonzero) torsion. 
\end{thm}

Theorem \ref{thmintro} implies that each integrally closed ideal, $\m$-full ideal, weakly $\m$-full ideal, and each ideal of the form $\m K$, where $K$ is an ideal of $R$, (in particular each power of the maximal ideal $\m$) satisfies the torsion conclusion of Conjecture \ref{HWC}; see Proposition \ref{propintclosed} and Corollary \ref{cor4}. Therefore, Theorem \ref{thmintro} recovers and extends the main theorems of
\cite{CGTT} and \cite{CT} concerning Conjecture \ref{HWC}; see \cite[2.17]{CGTT} and \cite[1.5]{CT}. 


An ideal $I$ defined as in Theorem \ref{thmintro} is a Burch ideal. The notion of a Burch ideal is a newcommer which has been recently defined and studied by Dao, Kobayashi and Takahashi in \cite{DKT}; see Definitions \ref{wmf}, \ref{burch} and Proposition \ref{prop4}. Burch ideals enjoy interesting properties and they are of interest to us. Therefore, besides making progress on Conjecture \ref{HWC}, we examine the relationship between the ideals we introduce in Theorem \ref{thmintro} and Burch ideals; see Section 3. Furthermore, in section 4, we construct various new examples, including examples of Burch ideals which are not weakly $\m$-full in the sense of \cite[3.7]{CIST}; see also Definition \ref{wmf}.


\section{Remarks, Main Results and Corollaries}
Recall that an ideal $I$ of a local ring $(R, \m)$ is called \emph{weakly $\m$-full} \cite{CIST} provided that $I=(\m I:_R\m)$. In \cite{CIST} weakly $\m$-full ideals were studied; there it was proved that each weakly $\m$-full ideal satisfies the torsion conclusion of Conjecture \ref{HWC}. The aim of this paper is to generalize the notion of a weakly $\m$-full ideal, and introduce a new class of ideals which, in particular, support Conjecture \ref{HWC}. For that we define:

\begin{dfn}\label{wmf} Let $R$ be a local ring, and let $I$ and $J$ be ideals of $R$.
We say $I$ is \textit{weakly $\m$-full with respect to $J$} provided that $(I:_RJ)=(\m I:_R \m J)$.
\end{dfn}

Note, according to Definition \ref{wmf}, if $I$ is a proper ideal of $R$, then $I$ is weakly $\m$-full with respect to $R$ if and only if $I=(\m I:_R\m)$, i.e., $I$ is weakly $\m$-full in the sense of\cite[3.7]{CIST}. Hence  Definition \ref{wmf} can be viewed as an extension of the notion of a weakly $\m$-full ideal.

In this section we give some examples of weakly $\m$-full ideals with respect to other ideals. More precisely, we prove that weakly $\m$-full ideals, as well as integrally closed ideals over local rings of positive depth, are weakly $\m$-full with respect to $\m^s$ for each $s\geq 0$; see Remark \ref{l5} and Proposition \ref{propintclosed}. Subsequently, we prove our main results, namely Theorems \ref{thmnew} and \ref{t3}, and give several corollaries of our argument. To wit, Corollary \ref{cor1} is the main consequence: it establishes a Tor-rigidity property of the class of ideals considered in Theorem \ref{thmintro}. Similarly, Corollary \ref{cor2} proves the claim given in Theorem \ref{thmintro} and Corollary \ref{cor4} lists two new classes of ideals satisfying the torsion conclusion of Conjecture \ref{HWC}.

\subsection*{Some examples of weakly $\m$-full ideals with respect to other ideals.}

Examples of weakly $\m$-full ideals with respect to other ideals are abundant: for example, each weakly $\m$-full ideal -- besides being weakly $\m$-full with respect to $R$ -- is also weakly $\m$-full with respect to $\m^s$ for each $s\geq 0$. We record this fact next:

\begin{rmk} \label{l5} Let $(R,\m)$ be a local ring and let $I$ be an ideal of $R$. If $(I:_R\m^s)=(\m I:_R\m^{s+1})$ for some integer $s\geq 0$, then, for each integer $u \geq s$, it follows that 
\begin{equation}\tag{\ref{l5}.1}
(I:_R\m^u)=\big((I:_R\m^s):_R\m^{u-s}\big)=(\m I:_R\m^{s+1}):_R\m^{u-s}\big)=(I:_R\m^{u+1}).
\end{equation}
The second equality in (\ref{l5}.1) is due to the fact that, if $L$ is an ideal of $R$, and $a$ and $b$ are nonnegative integers, then the equality $\big((L:_R\m^a):_R\m^{b}\big)=(L:_R\m^{a+b})$ holds.

Note (\ref{l5}.1) implies that, if $I$ is weakly $\m$-full, then $I$ is weakly $\m$-full with respect to $\m^i$ for each integer $i\geq 0$. 
\pushQED{\qed} 
\qedhere
\popQED	
\end{rmk}

Another source of weakly $\m$-full ideals is the class of integrally closed ideals: Goto \cite[2.4]{Goto87} proved that, if $k$ is infinite, then non-nilpotent integrally closed ideals are $\m$-full, and hence are weakly $\m$-full. Next we obtain a similar result, which is independent of the finiteness of the residue field, over local rings of positive depth. To prove this, we make use of the next equality which follows by the definition of colon ideals.

\begin{rmk}\label{rmkm} If $(R,\m)$ is a local ring and $I$ is an ideal of $R$, then $\m I=\m(\m I:_R\m)$. 
\end{rmk}

Throughout $\soc(R)$ denotes the \emph{socle} of a local ring $(R,\m)$, i.e., $\soc(R)=(0:_R\m)$.

\begin{prop} \label{propintclosed} Let $(R,\m)$ be a local ring and let $I$ be an integrally closed ideal of $R$. Assume $\depth(R)>0$. Then $I$ is weakly $\m$-full. In particular, $I$ is weakly $\m$-full with respect to $\m^s$ for each integer $s\geq 0$.
\end{prop}

\begin{proof} Note that, since $R$ has positive depth, we have $\soc(R)=0$ so that $\m$ is a faithful $R$-module. Let $x \in (\m I:_R\m)$. Then it follows that $\m I\subseteq \m (I+(x)) \subseteq \m (\m I:_R\m)=\m I$, where the last equality is due to Remark \ref{rmkm}. Therefore we conclude that $\m I=\m (I+(x))$. Now, by applying the determinantal trick (see, for example, \cite[1.8]{HS}),
we deduce that $x$ is integral over $I$. Hence, since $I$ is integrally closed, $x$ belongs to $I$. This proves $(\m I:_R\m)\subseteq I$, i.e. $(\m I:_R\m)=I$ and $I$ is weakly $\m$-full.
\end{proof}

\subsection*{Main results.} Next we prove our main results, which are Theorem \ref{thmnew} and Theorem \ref{t3}. Along the way, we also observe a result concerning Tor-rigidity; see Proposition \ref{propTor}.

In the following $\syz^t_R(M)$ denotes the $t$-th \emph{syzygy} module of $M$, which is  the image of the $t$-th differential map in a minimal free resolution of $M$. Note that $\syz^t_R(M)$ is uniquely determined up to isomorphism, since so is a minimal free resolution of $M$.

\begin{thm} \label{thmnew} Let $(R,\m)$ be a local ring, $M$ be a finitely generated $R$-module, $I$ and $J$ be ideals of $R$, and let $t\ge 0$ be an integer.
Assume the following conditions hold:
\begin{enumerate}[\rm(i)]
\item Either $J=R$ or $J$ is $\m$-primary, and $J \cdot \Omega_R^t(M)=0$.
\item $\soc(R)\subseteq I\subseteq \m(J:_R\m)$.
\end{enumerate}
If $\Tor^R_t(M,R/I)=0$, then it follows that $\pd_R(M)<t$.
\end{thm}

\begin{proof} Note, if $t=0$, then the assumption $\Tor^R_t(M,R/I)=0$ implies $M=0$ so that we have $\pd_R(M)<t$. Hence we assume $t\geq 1$, and proceed by induction on the length $\ell_R(R/J)$ of the $R$-module $R/J$.

If $\ell_R(R/J)=0$, then $J=R$ so that $J \cdot \Omega_R^t(M)= \Omega_R^t(M)=0$ and hence $\pd_R(M)<t$. So we assume $\ell(R/J)\geq 1$. In particular, as $J\neq R$, we assume $J$ is an $\m$-primary ideal of $R$. 
Therefore we have $0\neq \soc(R/J)=(J:_R\m)/J$ and hence $J\subsetneqq (J:_R\m)$. This implies $\ell_R(R/J)-\ell_R(R/(J:\m))=\ell_R((J:_R\m)/J)\geq 1$ so that $\ell(R/(J:\m))<\ell(R/J)$. Now our 
aim is to replace the ideal $J$ with $(J:_R\m)$ and conclude by the induction hypothesis that $\pd_R(M)<t$. For that we need to observe that the ideal $(J:_R\m)$ satisfies the same hypotheses as $J$ does.

Claim 1. The ideal $(J:_R\m)$ is $\m$-primary or equals $R$.\\
Proof of Claim 1. The claim is clear since $J$ is $\m$-primary and $J\subseteq (J:_R\m)$.

Claim 2. We have that $I\subseteq \m((J:_R\m):_R\m)$.\\
Proof of Claim 2. As $J\subseteq (J:_R\m)$, it follows $(J:_R\m) \subseteq ((J:_R\m):_R\m)$ and hence that $I\subseteq \m(J:_R\m)\subseteq \m((J:_R\m):_R\m)$, as claimed.

Claim 3. We have that $(J:_R\m) \cdot \Omega_R^t (M)=0$.\\
Proof of Claim 3. Set $N=\Omega_R^t(M)$, and consider a minimal free resolution of $M$:
$$F:\cdots \to F_{t+1} \xrightarrow[]{\partial_{t+1}} F_t \xrightarrow[]{\partial_{t}} F_{t-1} \to \cdots \to F_0 \to 0.$$ 
Then there is a short exact sequence of $R$-modules of the form:
\begin{equation}\tag{\ref{thmnew}.1}
0\to N \xrightarrow[]{\alpha} F_{t-1} \to \Omega_R^{t-1}(M) \to 0.
\end{equation}
As $\m(J:_R\m)\subseteq J$ and $J\cdot N=0$, we have $\m \alpha((J:_R\m) \cdot N)=\alpha(\m(J:_R\m) \cdot N)\subseteq\alpha(J \cdot N)=0$ so that $\alpha((J:_R\m) \cdot N)\subseteq \soc (R) \cdot F_{t-1}$.

We tensor the short exact sequence (\ref{thmnew}.1) with $R/I$ and obtain the exact sequence $\Tor^R_1(\syz^{t-1} M,R/I)\to N/I\cdot N \xrightarrow[]{\alpha\otimes_R R/I} F_{t-1}/I \cdot  F_{t-1}$. We conclude, since $\Tor^R_t(M,R/I)$, which is isomorphic to $\Tor^R_1(\syz^{t-1} M,R/I)$, vanishes, that $\alpha\otimes_R R/I$ is injective. 

As $\alpha((J:_R\m) \cdot N)\subseteq \soc (R) \cdot F_{t-1}$ and $\soc(R)\subseteq I$, it follows that
\begin{equation}\tag{\ref{thmnew}.2}
(\alpha\otimes_R R/I)\big((J:_R\m)(N/I \cdot N)\big) \subseteq \soc(R) (F_{t-1}/I \cdot F_{t-1})=0. 
\end{equation}

It now follows from (\ref{thmnew}.2) that $(J:_R\m)(N/IN)=0$, i.e., $(J:_R\m)N\subseteq I \cdot N$, because the map $\alpha\otimes_R R/I$ is injective. This fact, along with our assumption $I\subseteq \m(J:_R\m)$, yields
\begin{equation}\tag{\ref{thmnew}.3}
I \cdot N \subseteq \m(J:_R\m) \cdot N\subseteq (J:_R\m) \cdot N \subseteq I \cdot N \text{ so that }\m(J:_R\m) \cdot N=(J:_R\m) \cdot N.
\end{equation}
So (\ref{thmnew}.3) and Nakayama's lemma implies $(J:_R\m) \cdot \Omega_R^t(M)=0$, as claimed. Consequently, in view of Claim 1, 2 and 3, we deduce by the induction hypothesis that $\pd_R(M)<t$.
\end{proof}

The hypothesis $\soc(R)\subseteq I$ in Theorem \ref{thmnew} is necessary; we point out that fact with an example in Remark \ref{rmksocle} but first we record a preliminary result which seems to be of independent interest as it concerns Tor-rigidity. Recall that a finitely generated $R$-module $M$ is called \emph{Tor-rigid} \cite{Au} provided that the vanishing of $\Tor_1^R(M, N)$ yields the vanishing of $\Tor_2^R(M, N)$ for each
$R$-module $N$.

\begin{prop}\label{propTor} Let $(R, \m)$ be a local ring, which is not a field, and let $I$ be a nonzero ideal of $R$ such that $\soc(R)\nsubseteq I$ (so that $\soc(R)\neq 0$, or equivalently $\depth(R)=0$.) Let $y\in R$ such that $I\notni y \in \soc(R)$. Then we have $\Tor_1^R(R/yR, R/I)=0\neq \Tor_2^R(R/yR, R/I)$ and $\pd_R(R/yR)=\infty$. In particular, the $R$-module $R/yR$ is not Tor-rigid.
\end{prop}

\begin{proof} Notice, since $\m \neq 0$ and $y \cdot \m=0$, we see that $y$ cannot be a non zero-divisor on $R$. Hence it is clear that $\pd_R(R/yR)=\infty$: otherwise, the ideal $yR$ contains a non zero-divisor on $R$, which implies $y$ is also a non zero-divisor on $R$; see \cite[1.2.7(2)]{Av} (one can also observe $(0:_Ry)=\m$ and hence it follows that $yR \cong k$.)

Now let $x\in yR\cap I$. Then it follows that $I\ni x=ay$ for some $a\in R$. If $a \notin \m$, then $a$ is unit and hence $y\in I$. Therefore, we see $a\in \m$ so that $0=x=ay \in \m \cdot y=0$. Hence we conclude $yR\cap I=0$ so that $\Tor_1^R(R/yR, R/I)=0$. On the other hand,  since $0\neq I \neq R$, we have that $\Tor_2^R(R/yR, R/I) \cong \Tor_1^R(yR, R/I) \cong \Tor_1^R(k, R/I) \neq 0$.
\end{proof}

\begin{rmk} \label{rmksocle} Let $(R,\m)$ be a local ring and let $I$ be an ideal of $R$. Assume $0\neq \soc(R) \nsubseteq I$ and pick an element $I\notni y \in \soc(R)$. Following the notations of Theorem \ref{thmnew}, we set $J=\m$, $M=R/yR$ and $t=1$ (e.g., let $R=k[\![x,y]\!]/(y^2, xy)$, $xR=I\notni y \in \soc(R)=yR$ and $M=R/yR$.) Then $J\syz M=0$ and $I\subseteq \m=\m R=\m(\m:_R\m)=\m(J:_R\m)$. However we have $\Tor_1^R(R/yR, R/I)=0\neq \Tor_2^R(R/yR, R/I)$ and $\pd_R(R/yR)=\infty$; see Proposition \ref{propTor}.
\end{rmk}

Next is the second main theorem of this paper: we use it along with Theorem \ref{thmnew} for the proof of Corollary \ref{cor1}. As mentioned previously, Corollary \ref{cor1} implies that the ideals considered in Theorem \ref{thmintro} enjoy a Tor-rigidity property.

\begin{thm} \label{t3} Let $(R,\m)$ be a local ring, and let $I$ and $J$ be ideals of $R$. Assume the following conditions hold:
\begin{enumerate}[\rm(i)]
\item $0\neq I \subseteq \m J$.
\item $(I:_RJ)$ is $\m$-primary.
\item $I$ is weakly $\m$-full with respect to $J$, i.e., $(I:_RJ)=(\m I:_R \m J)$.
\end{enumerate}
If $\Tor^R_t(M,R/I)=0$ for some finitely generated $R$-module $M$ and some integer $t\geq 0$, then $\Omega_R^{t}(M)$ is annihilated by $J$, i.e., $J\cdot \Omega_R^{t}(M)=0$.
\end{thm}

\begin{proof}
Note, if $t=0$, then the assumption $\Tor^R_t(M,R/I)=0$ implies $M=0$; hence we may assume $t>0$.

Let $F:\cdots \to F_{t+1} \xrightarrow[]{\partial_{t+1}} F_t \xrightarrow[]{\partial_{t}} F_{t-1} \to \cdots \to F_0 \to 0$ be an minimal free resolution of $M$ over $R$.
Tensoring $F$ with $R/I$, we obtain the following complex over $R/I$:
$$
\overline{F}:\cdots \to \overline{F}_{t+1} \xrightarrow[]{\overline{\partial}_{t+1}} \overline{F}_t \xrightarrow[]{\overline{\partial}_{t}} \overline{F}_{t-1} \to \cdots \to \overline{F}_0 \to 0.
$$
Note that, since $\Tor^R_t(M,R/I)=0$, we have $\im(\overline{\partial}_{t+1})=\ker(\overline{\partial}_t)$. We proceed by establishing the following claim.

Claim: $J$ annihilates $\im(\overline{\partial}_{t})$, i.e., $J\cdot \im (\overline{\partial}_{t})=0$.\\
Proof of the claim. Suppose the claim is not true, i.e., suppose $J\cdot \im (\overline{\partial}_{t})\neq 0$, and look for a contradiction.

First note that
\begin{align}
\tag{\ref{t3}.1}  J\cdot \im (\overline{\partial}_{t})\neq 0  \Longrightarrow J \cdot \im(\partial_t)\not\subseteq I \cdot F_{t-1} \Longrightarrow \im(\partial_t) \nsubseteq (I:_RJ) \cdot F_{t-1}.
\end{align}

The hypothesis $0\not=I\subseteq \m J$ implies that $(I:_RJ)\neq R$; hence there exists an integer $u\geq 1$ such that $\m^u \subseteq (I:_RJ)$ since $(I:_RJ)$ is $\m$-primary. Now, in view of (\ref{t3}.1), we
can choose such an integer $u$ so that
\begin{align}
\tag{\ref{t3}.2}   \m^u \cdot \im(\partial_t)\subseteq (I:_RJ) \cdot F_{t-1} \text{ and } \m^{u-1} \cdot \im(\partial_t)\nsubseteq (I:_RJ) \cdot F_{t-1}. 
\end{align}
Pick an element $a \in F_{t-1}$ such that $(I:_RJ) \cdot F_{t-1} \notni a\in \m^{u-1} \cdot \im(\partial_t)$. Then it follows
\begin{align}
\tag{\ref{t3}.3} \m a \subseteq \m^u \cdot \im(\partial_t) \subseteq (I:_RJ) \cdot F_{t-1} \Longrightarrow \m J \cdot a \subseteq I \cdot F_{t-1} \Longrightarrow m J \cdot \overline{a} =0 \text{ in } \overline{F}_{t-1}.
\end{align}
Moreover, we have
\begin{align}
\notag{} a \in \m^{u-1} \cdot \im(\partial_t) \subseteq  \im(\partial_t) & \Longrightarrow a=\partial_t(b) \text{ for some } b\in F_t \\ \notag{} & \Longrightarrow \overline{\partial_t}(\m J \cdot \bar{b} )= \m J \cdot \overline{\partial_t}(\bar{b}) = \m J \cdot \overline{a} =0 \text{ in } \overline{F}_{t-1} \\ \notag{} & \Longrightarrow \m J \cdot \bar{b}  \subseteq  \ker(\overline{\partial}_{t})= \im(\overline{\partial}_{t+1}) \\ \notag{} & \Longrightarrow \m J \cdot b \subseteq \im(\partial_{t+1}) + I\cdot F_{t} \\ \notag{} & \Longrightarrow \partial_t(\m J \cdot b) \subseteq \partial_t \big( \im(\partial_{t+1}) + I\cdot F_{t} \big) \\ \notag{} & \Longrightarrow \m J \cdot \partial_t(b) \subseteq I \cdot \partial_t (F_{t})  \\ \notag{} & \Longrightarrow \m J \cdot \partial_t(b) \subseteq \m I \cdot F_{t-1} \\ \notag{} & \Longrightarrow \m J \cdot a \subseteq \m I \cdot F_{t-1} \\ \notag{} & \Longrightarrow a \in (\m I:_R\m J)\cdot F_{t-1} = (I:_RJ)\cdot F_{t-1},
\end{align}
which is a contradiction. This establishes the claim.

Now notice, in view of the claim, it follows that
\begin{align}
\notag{} \overline{\partial}_{t}(J \cdot \overline{F}_t)=J \cdot \overline{\partial}_{t}(\overline{F}_t)=J \cdot \im( \overline{\partial}_{t})=0  & \Longrightarrow  J \cdot \overline{F}_t \in \ker(\overline{\partial}_{t})=\im(\overline{\partial}_{t+1})  \\ \notag{} & \Longrightarrow J \cdot F_t\subseteq \im(\partial_{t+1})+I \cdot F_t \\ \tag{\ref{t3}.4} & \Longrightarrow J \cdot F_t = \big(\im(\partial_{t+1})+I \cdot F_t \big) \cap (J \cdot F_t )
\\ \notag{} & \Longrightarrow J \cdot F_t = \big(\im(\partial_{t+1}) \cap (J \cdot F_t)  \big)+ I \cdot F_t,
\end{align}
where the last implication in (\ref{t3}.4) is due to the modular law since $I\subseteq \m J \subseteq J$. On the other hand, we have
\begin{align}
\tag{\ref{t3}.5} \big(\im(\partial_{t+1}) \cap (J \cdot F_t)  \big)+ I \cdot F_t \subseteq \big(\im(\partial_{t+1}) \cap (J \cdot F_t)  \big)+\m \cdot (J\cdot F_{t}) \subseteq J \cdot F_t.
\end{align}
Therefore, by (\ref{t3}.4) and (\ref{t3}.5), we conclude
\begin{align}
\notag{} J \cdot F_t=\big(\im(\partial_{t+1}) \cap (J \cdot F_t)  \big)+\m \cdot (J\cdot F_{t}) & \Longrightarrow  J \cdot F_t= \im(\partial_{t+1}) \cap (J \cdot F_t) \\ \notag{} & \Longrightarrow J \cdot F_t \subseteq \im(\partial_{t+1}) = \ker(\partial_{t}) \\ \tag{\ref{t3}.6} & \Longrightarrow \partial_{t}(J \cdot F_{t})=0 \\ \notag{} & \Longrightarrow J \cdot \partial_{t}(F_{t})=  J \cdot \Omega_R^t(M) =0.
\end{align}
Here, in (\ref{t3}.6), the first implication follows from Nakayama's lemma.
\end{proof}

\subsection*{Some corollaries of Theorems \ref{thmnew} and \ref{t3}.} This subsection contains corollaries of our main results, Theorems \ref{thmnew} and \ref{t3}, with Corollary \ref{cor1} being the primary workhorse. Corollary \ref{cor2} establishes the claim of Theorem \ref{thmintro} and Corollary \ref{cor4} lists two new classes of ideals which satisfy the torsion conclusion of Conjecture \ref{HWC}. We start with a remark needed for the proof of Corollary \ref{cor1}.

\begin{rmk} \label{rmkcor} Let $(R,\m)$ be a local ring, and let $I$ and $J$ be ideals of $R$ such that $I\subseteq J$. Then the following conditions are equivalent:
\begin{enumerate}[\rm(i)]
\item $J$ is $\m$-primary and $(I:_{R}J)$ is $\m$-primary.
\item $I$ is $\m$-primary.
\end{enumerate}
It is clear, since $I\subseteq J$ and $I\subseteq (I:_{R}J)$, that (ii) implies (i). On the other hand, if $(I:_{R}J)$ is $\m$-primary, then $\m^i J \subseteq  I$ for some $i\geq 1$. Hence, if $J=R$, then $I$ is $\m$-primary. Moreover, if $J\neq R$ but $J$ is $\m$-primary, then $\m^{j} \subseteq J$ and so $\m^{i+j}\subseteq \m^{i}J\subseteq I$, i.e., $I$ is $\m$-primary.
\end{rmk}

The first corollary we give is the main consequence of our argument: 

\begin{cor} \label{cor1} Let $(R,\m)$ be a local ring, and let $I$ and $J$ be ideals of $R$. Assume:
\begin{enumerate}[\rm(i)]
\item $I$ is weakly $\m$-full with respect to $J$, i.e., $(I:_RJ)=(\m I:_R \m J)$.
\item $I$ is $\m$-primary and $0\neq I \subseteq \m J$.
\item $\soc (R)\subseteq I$ (e.g., $\depth(R)>0$).
\end{enumerate}
If $\Tor^R_t(M,R/I)=0$ for some finitely generated $R$-module $M$ and some integer $t\geq 0$, then it follows that $\pd_R(M)<t$.
\end{cor}

\begin{proof} It follows from Theorem \ref{t3} that $J \cdot \syz^t M=0$. Also we see $0\neq I \subseteq \m J \subseteq \m (J:_R \m)$ since $0\neq I \subseteq \m J$.
Hence the conclusion follows from Theorem \ref{thmnew} since both $J$ and $(I:_{R}J)$ are $\m$-primary ideals; see Remark \ref{rmkcor}.
\end{proof}

Corollary \ref{cor1}, in view of \cite[2.11, 2.13 and 2.15]{CGTT}, yields a prompt proof for Theorem \ref{thmintro}. Here we include an argument for the completeness and the convenience of the reader, and establish the claim of Theorem \ref{thmintro} in Corollary \ref{cor2}. 

In the proofs of Corollaries \ref{cor2} and \ref{cor3}, $\Tr_R(I)$ denotes the Auslander transpose of the ideal $I$ of the ring $R$; see \cite[2.5]{Au}.

\begin{cor} \label{cor2} Let $(R, \m)$ be a local ring, and let $I$ and $J$ be ideals of $R$. Assume $R$ is not regular and $\depth(R)=1$. Assume further $0\neq I\subseteq \m J$ and $I$ is weakly $\m$-full with respect to $J$, i.e., $(I:_RJ)=(\m I:_R \m J)$. Then $I\otimes_RI^{\ast}$ has (nonzero) torsion.
\end{cor}

\begin{proof} First note that $I\otimes_RI^{\ast}\neq 0$: if $I\otimes_RI^{\ast}=0$, then, since $I\neq 0$, we see $I^{\ast}=0$. However, this is equivalent to $I$ being a torsion $R$-module, which forces $I=0$ as $I$ is a torsion-free $R$-module. Therefore we have that $I\otimes_RI^{\ast}\neq 0$.

Next suppose $I\otimes_RI^{\ast}$ is torsion-free. Then, as $\Tor_1^R(R/I, I^{\ast})$ is torsion, by tensoring the short exact sequence $0\to I \to R \to R/I \to 0$ with $I^{\ast}$ over $R$, we see $\Tor_1^R(R/I, I^{\ast})=0$. Hence it follows from Corollary \ref{cor1} that $\pd_R(I^{\ast})<1$, i.e., $I^{\ast}$ is free. This implies $I$ has rank so that $\Tr_R(I)_{\p}=0$ for each associated prime ideal $\p$ of $R$. Consequently,  $\Ext^1_R(\Tr_R(I), R)$ is a torsion $R$-module which vanishes due to the exact sequence: $0\to \Ext^1_R(\Tr_R(I), R) \to I \to I^{\ast\ast} \to \Ext^2_R(\Tr_R(I), R)$; see \cite[2.6(a)]{Au}. Furthermore, since $\pd_R(\Tr_R(I))\leq 1$, we conclude that $\Ext^2_R(\Tr_R(I), R)=0$. Therefore, $I\cong I^{\ast\ast}$ and thus $I$ is a free $R$-module. However this, along with Corollary \ref{cor1}, implies that $R$ is regular. So $I\otimes_RI^{\ast}$ has (nonzero) torsion.
\end{proof}

It seems worth noting that, if $R$ is a local ring of positive depth and $I$ is an ideal as in Corollary \ref{cor1}, then the vanishing of a single $\Ext^i_R(I,I)$ or $\Tor_i^R(I,I)$ forces $R$ to be regular.

\begin{cor} \label{cor3} Let $(R, \m)$ be a local ring such that $\depth(R)>0$ and $R$ is not regular. Let $I$ and $J$ be ideals of $R$ such that $0\neq I\subseteq \m J$ and $(I:_RJ)=(\m I:_R \m J)$. Then it follows that $\Ext^i_R(I,I)\neq 0\neq \Tor_i^R(I,I)$ for each $i\geq 0$.
\end{cor}

\begin{proof} Note, as $R$ is not regular, the claim about the non-vanishing of $\Tor_i^R(I,I)$  is due to Corollary \ref{cor1}.

Now suppose $\Ext^n_R(I,I)=0$ for some $n\geq 1$, i.e., $\Ext^{1}_R(\Omega_R^{n-1}(I), \Omega_R(R/I))=0$. Hence \cite[2.6(i)]{CIST} implies that $\Tor_2^R(\Tr_R(\Omega_R (\Omega_R^{n-1}(I))), R/I)=0$, i.e., $\Tor_2^R(\Tr_R(\Omega^{n}_R(I)), R/I)$ vanishes. This shows that $\pd_R(\Tr_R(\Omega_R^n (I)))<2$; see Corollary \ref{cor1}. We know, as syzygy modules are torsionless, that $\Ext^1_R(\Tr_R(\Omega_R^n (I)), R)=0$. Therefore $\Tr_R(\Omega_R^n (I))$ is free. Consequently $\Omega^n_R(I)$ is free, i.e., $\pd_R(I)<\infty$. However this would imply that $R$ is regular; see Corollary \ref{cor1}.
\end{proof}

Our next aim is to state two new classes of ideals in Corollary \ref{cor4} that satisfy the torsion conclusion of Conjecture \ref{HWC}; such ideals are also Tor-rigid and can be used to test the finiteness of the projective dimenison of modules via the vanishing of Tor. The following lemma plays an important role for the proof.

\begin{lem} \label{lwmf}
Let $(R,\m)$ be a local ring, and let $I$ and $J$ be ideals of $R$ such that $0\neq I=\m J$. Let $K$ be an ideal of $R$ such that $J\subseteq K \subseteq (I:_R\m)$. Then it follows that 
$I$ is weakly $\m$-full with respect $K$ and $(I:_RK)=\m$.
\end{lem}

\begin{proof} Note that we have $I=\m J\subseteq \m K \subseteq \m(I:_R\m)$ hold. Moreover, it follows that $$\m(I:_R\m)=\m(\m J:_R\m)=\m J=I.$$
Here the second equality is due to Remark \ref{rmkm}. Therefore we see that $I=\m K$. Since $I$ is nonzero, we obtain $\m \subseteq (I:_RK)\neq R$, i.e., $(I:_RK)=\m$. As the equality $\m=(\m I:_RI)$ holds by definition, we conclude $(I:_RK)=\m=(\m I:_RI)=(\m I:_R\m K)$. This shows that $I$ is weakly $\m$-full with respect to $K$.
\end{proof}

Part (ii) and part (iv) of Corollary \ref{cor4} are due to Celikbas, Goto, Takahashi and Taniguchi \cite[2.17]{CGTT}, and Celikbas and Takahashi \cite[1.5]{CT}, respectively. Here we establish the claims in part (i) and part (iii) of the corollary as a consequence of our argument, and record parts (ii) and (iv) for the completeness; note that part (i) is an extension of part (ii), and part (iii) is an extension of part (iv).

\begin{cor} \label{cor4} Let $(R,\m)$ be a non-regular local ring such that $\depth(R)>0$, and let $I$ be an ideal of $R$ satisfying at least one of the following conditions:
\begin{enumerate}[\rm(i)]
\item $I$ is $\m$-primary, weakly $\m$-full with respect to $\m^s$ for some $s\geq 0$, and $I\subseteq \m^{s+1}$.
\item $I$ is $\m$-primary and weakly $\m$-full. 
\item $I=\m J$ for some $\m$-primary ideal $J$ of $R$.
\item $I=\m^s$ for some $s\geq 1$.
\end{enumerate}
Then the following hold:
\begin{enumerate}[\rm(a)]
\item If $\Tor^R_t(M,R/I)=0$ for some finitely generated $R$-module $M$ and some integer $t\geq 0$, then it follows that $\pd_R(M)<t$.
\item If $R$ is a one-dimensional domain, then  $I\otimes_RI^{\ast}$ has (nonzero) torsion.
\end{enumerate}
\end{cor}

\begin{proof} Note that, $I$ is $\m$-primary in each case; in particular, $I$ is nonzero and proper. Note also that, since $\depth(R)>0$, we have $\soc(R)=0$. Hence part (i) follows from Corollaries \ref{cor1} and \ref{cor2}. Part (ii) is a special case of part (i), namely the case where $s=0$. Similarly part (iv) is a special case of part (iii), which we establish next.

Assume $I=\m J$ for some $\m$-primary ideal $J$ of $R$ and set $K=(I:_R\m)$. Then Lemma \ref{lwmf} implies that $K\neq 0$ and $I$ is weakly $\m$-full with respect $K$. Moreover, it follows from Remark \ref{rmkm} that $I=\m J=\m(\m J:_R\m)=\m(I:_R\m)=\m K$. So the claims follow from Corollaries \ref{cor1} and \ref{cor2}.
\end{proof}

If we consider the ideals stated in parts (i) and (iii) of Corollary \ref{cor4} over local rings $R$ which do not necessarily have positive depth, then we can conclude, by the vanishing of $\Tor^R_t(M,R/I)$ for some finitely generated $R$-module $M$ and some integer $t\geq 0$, that the syzygy module $\Omega_R^t(M)$ is annihilated by a power of the maximal ideal, or by the ideal $(I:_{R}\m)$. This result does not directly address Conjecture \ref{HWC}, but we prove it here as it may be helpful for a further separate study of the class of ideals in question.

\begin{cor} \label{corsecondary} Let $(R,\m)$ be a local ring, $I$ an ideal of $R$ and let $M$ be a finitely generated $R$-module such that $\Tor^R_t(M,R/I)=0$ for some integer $t\geq 0$. 
\begin{enumerate}[\rm(i)]
\item If $I$ is $\m$-primary, weakly $\m$-full with respect to $\m^s$ for some $s\geq 0$, and $0\neq I\subseteq \m^{s+1}$, then it follows that $\m^s \cdot \Omega_R^t (M)=0$.
\item If $I=\m J$ for some ideal $J$ of $R$, then it follows that $(I:_{R}\m)\cdot \Omega_R^t (M)=0$.
\end{enumerate}
\end{cor}

\begin{proof} For part (i), note that $(I:_R\m^s)$ is $\m$-primary since $I\subseteq (I:_R\m^s)$. Hence, setting $J=\m^s$, we conclude from Theorem \ref{t3} that $J \cdot \Omega_R^t(M)=\m^s \cdot \Omega_R^{t}(M)=0$.

For part (ii), set, as in Corollary \ref{cor4}, that $K=(I:_R\m)$. If $I=0$, then $K=\soc(R)$ and hence the equality $K \cdot \syz_R^t (M)=0$ holds since $\Omega^t_R(M)\subseteq \m F$ for some free $R$-module $F$. So we assume $I\neq 0$. Then Lemma \ref{lwmf} implies that $I$ is weakly $\m$-full with respect $K$ and the ideal $(I:_RK)$ is $\m$-primary, more precisely $(I:_RK)=\m$. Therefore Theorem \ref{t3} shows that $0=K \cdot \Omega^t_R(M) = (I:_{R}\m) \cdot \Omega^t_R(M)$. 
\end{proof}


\section{Some relations between weakly $\m$-full and Burch ideals} 

The article \cite{DKT} introduces and studies a class of ideals, called \emph{Burch ideals}, over local rings. We recall the definition of such ideals next:

\begin{dfn} \label{burch} (\cite{DKT}) Let $(R,\m)$ be a local ring and let $I$ be an ideal of $R$. Then $I$ is called \textit{Burch} provided that $(I:_R\m)\not=(\m I:_R\m)$, i.e., $(I:_R\m)\nsubseteq(\m I:_R\m)$.
\end{dfn}

The definition of a Burch ideal is simple, but as shown in \cite{DKT}, Burch ideals enjoy remarkable ideal-theoretic and homological properties. For example, it is known that, if $I$ is a Burch ideal in a local ring $R$, then one has $\depth(R/I)=0$; see \cite[2.1]{DKT}. On the other hand, if $I$ is a weakly $\m$-full ideal of $R$ such that $\depth(R/I)=0$, then it follows that $I$ is Burch; see \cite[2.4]{DKT}. We give several examples of Burch ideals in section 4 and, in passing, we record a relation between Burch ideals and weakly $\m$-full ideals with respect to $(I:_R\m)$.

\begin{rmk} Let $(R,\m)$ be a local ring and let $I$ be an ideal of $R$. Then the following conditions are equivalent:
\begin{enumerate}[\rm(i)]
\item $I$ is weakly $\m$-full with respect to $(I:_R\m)$, i.e., $\big(I:_R(I:_R\m)\big)=\big(\m I:_R\m(I:_R\m)\big)$.
\item $I$ is Burch or $\depth(R/I)>0$.
\end{enumerate}
The equivalence of (i) and (ii) follows from these facts, which can be checked directly:
\begin{enumerate}[\rm(1)]
\item $\m \subseteq \big(I:_R(I:_R \m)\big)$
\item $\m = \big(I:_R(I:_R \m)\big)$ if and only if $\depth(R/I)=0$.
\item $I$ is Burch if and only if $(I:_R\m)\not=(\m I:_R\m)$ if and only if $\big(\m I:_R\m(I:_R\m)\big)\neq R$. $\qed$
\end{enumerate}
\end{rmk}

If $(R, \m)$ is a local ring such that $\depth(R)>0$, then an $\m$-primary ideal $I$, which is weakly $\m$-full, is Tor-rigid in the sense that, if $\Tor^R_t(M,R/I)=0$ for some finitely generated $R$-module $M$ and some integer $t\geq 0$, then it follows that $\pd_R(M)<t$; see \cite[3.3]{CHKV} or \cite[2.10]{CGTT}. A similar Tor-rigidity property, albeit one that requires the vanishing of two $\Tor$ modules, was obtained by L. Burch for Burch ideals: 

\begin{thm}\label{BTor}(\cite[Theorem 5(ii), page 949]{BurchL}) Let $(R,\m)$ be a local ring and let $I$ be an ideal of $R$. Assume $I$ is Burch. 
If $\Tor^R_t(M,R/I)=0=\Tor^R_{t+1}(M,R/I)$ for some finitely generated $R$-module $M$ and some integer $t\geq 0$, then $\pd_R(M)\leq t$.
\end{thm}

In view of the fact that there are Burch ideals which are not weakly $\m$-full (see, for example, Examples \ref{e4} and \ref{notweaklymfull}), the foregoing discussion raises the following question:

\begin{ques} \label{burchq} Let $(R,\m)$ be a local ring and let $I$ be an ideal of $R$. Assume $\depth(R)>0$. Assume further $I$ is $\m$-primary and Burch.
If $\Tor^R_t(M,R/I)=0$ for some finitely generated $R$-module $M$ and some integer $t\geq 0$, then does it follows $\pd_R(M)<\infty$?
\end{ques}

Question \ref{burchq} simply asks whether it is possible to improve Theorem \ref{BTor} when the ring considered has positive depth and the Burch ideal in question is $\m$-primary. We do not know whether or not Question \ref{burchq} is true in general, but the positive depth assumption is necessary for the question; see Example \ref{e2}. It is worth noting that an affirmative answer to Question \ref{burchq} implies that  each Burch ideal satisfies the torsion conclusion of Conjecture \ref{HWC}; see the proof of Corollary \ref{cor2}.

Note that, Corollary \ref{cor4}(iii) provides an affirmative answer to Question \ref{burchq} for a class of Burch ideals, namely for nonzero ideals of the form $\m J$ for some ideal $J$ of $R$; see \cite[2.2 (2)]{DKT}. The aim of this section is to prove Proposition \ref{prop4}, which shows that the ideals considered in Theorem \ref{thmintro} are Burch ideals. Consequently, Proposition \ref{prop4}, in conjunction with Corollary \ref{cor1}, yield another affirmative answer to Question \ref{burchq} for a special class of Burch ideals: more precisely, we see that the class of Burch ideals considered in Theorem \ref{thmintro} satisfy the projective dimension conclusion of Question \ref{burchq}.

Next we recall the definition of Loewy length and then prepare a lemma for our proof of Proposition \ref{prop4}.

\begin{dfn} \label{dfnll}
Let $(R,\m)$ be a local ring, $I$ be an ideal of $R$, and let $M$ be an $R$-module. Then the \textit{Loewy length} $\ll(M)$ of $M$ is $\inf\{s\mid \m^s\cdot M=0\}$.
\end{dfn}

\begin{lem} \label{l2} If $I$ and $J$ are ideals in a local ring $(R,\m)$, then the following hold:
\begin{enumerate}[\rm(i)]
\item $\ll\big( R/(I:_RJ) \big)=\inf\{s\mid \m^s \cdot J\subseteq I\}$. 
\item If $J\not\subseteq I$, then $\ll\big(R/(I:_RJ)\big)=\ll\big(R/(I:_R\m J)\big)+1=\ll\big(R/\big((I:_R\m):_R J\big)\big)+1$.
\item If $J\not\subseteq I$, $(I:_RJ)$ is $\m$-primary, and $\ll\big(R/(\m I:_RJ)\big)=\ll\big(R/(I:_RJ)\big)+1$,
then $I$ is Burch.
\end{enumerate}
\end{lem}

\begin{proof} Part (i) follows by the definition of Loewy length.

We assume, for part (ii), that $J\not\subseteq I$. Then note that:
\begin{align}
\notag{} \ll\big(R/(I:_R\m J)\big)=\inf\{t\mid \m^t \cdot (\m J) \subseteq I \} = & \inf\{u \mid \m^u \cdot J \subseteq (I:_R\m)\}  \\ \notag{} = & \ll\big(R/\big((I:_R\m):_R J\big)\big).
\end{align}
Moreover, it follows that:
\begin{align}
\notag{} \m^s \cdot J \subseteq I \text{ if and only if } \m^{s-1} \cdot (\m J) \subseteq I \text{ if and only if } \m^{s-1} \cdot J \subseteq (I:_R\m).
\end{align}
These observations establish part (ii).

For part (iii), we assume $J\not\subseteq I$ and note:
\begin{align}
\tag{\ref{l2}.1} \inf\{s\mid \m^s \cdot J\subseteq I\} = \ll\big( R/(I:_RJ) \big)= &\ll\big(R/\big((I:_R\m):_R J\big)\big)+1 \\ \notag{} = &  \inf\{r \mid \m^r \cdot J \subseteq (I:_R\m)\} +1.
\end{align}
Here, in (\ref{l2}.1), the first equality follows from part (i) and the second equality is due to part (ii).

As $J\not\subseteq I$, we have that $J\not \subseteq \m I$; therefore we can use parts (i) and (ii) and obtain the equalities stated in (\ref{l2}.1) for the pair $(\m I, J)$:
\begin{align}
\tag{\ref{l2}.2} \inf\{s\mid \m^s \cdot J\subseteq \m I\} = \ll\big( R/(\m I:_RJ) \big)= &\ll\big(R/\big((\m I:_R\m):_R J\big)\big)+1 \\ \notag{} = &  \inf\{r \mid \m^r \cdot J \subseteq (\m I:_R\m)\} +1.
\end{align} 
Now we assume $(I:_RJ)$ is $\m$-primary, and $\ll\big(R/(\m I:_RJ)\big)=\ll\big(R/(I:_RJ)\big)+1$. Then,
\begin{align}
\tag{\ref{l2}.3} \inf\{r \mid \m^r \cdot J \subseteq (\m I:_R\m)\}  = & \ll\big(R/(\m I:_RJ)\big) -1 \\ \notag{} = &  \ll\big(R/(I:_RJ)\big)  \\ \notag{} = & \inf\{r \mid \m^r \cdot J \subseteq (I:_R\m)\} +1 
\end{align} 
The first and the third equality follows from (\ref{l2}.2) and (\ref{l2}.1), respectively; while the second equality is due to our assumption. Note that, since $(I:_RJ)$ is $\m$-primary, the infimum of both sets in (\ref{l2}.3) are finite integers. Hence (\ref{l2}.3) shows that $(I:_R\m) \not=(\m I:_R\m)$, i.e., $I$ is Burch. 
\end{proof}

We note and record, by letting $J=R$ in Lemma \ref{l2}(ii), a criterion for an $\m$-primary ideal to be Burch:

\begin{rmk} Let $(R,\m)$ be a local ring and let $I$ be an ideal of $R$. If $I$ is $\m$-primary and $\ll(R/\m I)=\ll(R/I)+1$, then $I$ is Burch.
\end{rmk}

We are now ready to prove that the class of ideals considered in Theorem \ref{thmintro} is contained in that of Burch ideals.

\begin{prop} \label{prop4} Let $(R,\m)$ be a local ring, and let $I$ and $J$ be ideals of $R$ such that
\begin{enumerate}[\rm(i)]
\item $0\neq I \subseteq \m J$.
\item $(I:_RJ)$ is $\m$-primary.
\item $I$ is weakly $\m$-full with respect to $J$, i.e., $(I:_RJ)=(\m I:_R \m J)$.
\end{enumerate}
Then $I$ is Burch.
\end{prop}

\begin{proof} Note, since $0\neq I \subseteq \m J$, it follows from Nakayama's lemma that $J\not\subseteq I$; in particular we have that $J\not\subseteq \m I$. Hence, in view of Lemma \ref{l2}(ii) applied to the pair $(\m I, J)$, we obtain $\ll\big(R/(\m I:_RJ)\big)=\ll\big(R/(\m I:_R\m J)\big)+1$. So, as we assume $(I:_RJ)=(\m I:_R \m J)$, we conclude that $\ll\big(R/(\m I:_RJ)\big)=\ll\big(R/(I:_R J)\big)+1$.  Consequently, since $(I:_RJ)$ is $\m$-primary, \ref{l2}(iii) implies that $I$ is Burch.
\end{proof}

We finish this section by proving that the hypothesis $0\neq I \subseteq \m J$  in Proposition \ref{prop4} can be relaxed for certain ideals $J$. More precisely, we prove in Proposition \ref{c1} that, if $J=\m^s$ for some $s$ with $0\le s <\ll(R/I)$ and $I$ is an ideal such that $I$ is weakly $\m$-full with respect to $J$ and $(I:_RJ)$ is $\m$-primary (or equivalently $I$ is $\m$-primary; see Remark \ref{rmkcor}), then $I$ is Burch. Our motivation for proving Proposition \ref{c1} also comes from our argument in section 4; see Example \ref{e6}.

\begin{prop} \label{c1} Let $(R,\m)$ be a local ring and let $I$ be an $\m$-primary ideal of $R$. If $I$ is weakly $\m$-full with respect to $\m^s$, i.e.,  
$(I:_R\m^s)= (\m I:_R\m^{s+1})$, for some integer $s$ where $0\le s <\ll(R/I)$, then $I$ is Burch. 
\end{prop}

The assertion of Proposition \ref{c1} follows by Lemma \ref{l2}(iii) and the following result:

\begin{lem} \label{l3} Let $(R,\m)$ be a local ring and let $I$ be an $\m$-primary ideal of $R$.
Then the following conditions are equivalent:
\begin{enumerate}[\rm(i)]
\item The equality $\ll(R/\m I)=\ll(R/I)+1$ holds.
\item $(I:_R\m^s) = (\m I:_R\m^{s+1})$ where $s=\ll(R/I)-1$.
\item $(I:_R\m^s) = (\m I:_R\m^{s+1})$ for some integer $s$ where $0\le s<\ll(R/I)$.
\end{enumerate}
\end{lem}

\begin{proof} Assume part (i) holds and set $s=\ll(R/I)-1$. Then it follows that $\m^s\not\subseteq I$ and $\m^{s+1}\subseteq I$. This shows $\m \subseteq (I:_R\m^s) \neq R$, i.e., $(I:_R\m^s)=\m$. So we see  $\m=(I:_R\m^s) \subseteq (\m I:_R\m^{s+1})$. As $\ll(R/\m I)=\ll(R/I)+1=s+2$, we have $\m^{s+1}\not\subseteq \m I$ so that $(\m I:_R\m^{s+1})\neq R$.
This implies that $(I:_R\m^s)=\m=(\m I:_R\m^{s+1})$, and hence establishes part (ii).

It is clear that part (ii) implies part (iii). Hence we assume part (iii), i.e., assume the equality $(I:_R\m^s) = (\m I:_R\m^{s+1})$ holds for some $s$ where $0\le s<\ll(R/I)$, and proceed to prove part (i). 

As $I\neq R$, it follows from Lemma \ref{l2}(ii) that $\ll(R/I)=\ll(R/(I:_R\m))+1$. Note that, if $\m^r \subseteq I$ for some $r$, where $0<r<\ll(R/I)$, then the equality $\ll(R/\m I)=\ll(R/I)+1$ holds by the definition of Loewy length. So, for all $r=1, \ldots, \ll(R/I)-1$, we assume $\m^r \nsubseteq I$. Hence a repeated application of Lemma \ref{l2}(ii) yields $\ll(R/I)=\ll(R/(I:\m^s))+s$. Similarly we obtain $\ll(R/\m I)=\ll(R/(\m I:\m^{s+1}))+s+1$. Consequently, since we have $(I:_R\m^s) = (\m I:_R\m^{s+1})$, we conclude that the equality $\ll(R/\m I)=\ll(R/I)+1$ holds.
\end{proof}


\section{Some examples of (non)weakly $\m$-full ideals}

This section is devoted to some examples that corroborate and complement our results from section 3. For example, we construct examples of Burch ideals which are not weakly $\m$-full, but are weakly $\m$-full with respect to $\m$ and $\m^3$; see Examples \ref{e4} and \ref{notweaklymfull}. We start by showing that the depth assumption in Theorem \ref{t3} and Corollary \ref{cor3} is necessary:

\begin{ex} \label{e2}
Let $R=k[x,y]/(x,y)^3$ and consider the ideal $I=y\m=(xy, y^2)$ of $R$.
It follows from Lemma \ref{lwmf} that $I$ is weakly $\m$-full with respect to $(I:_R\m)$, which is $(y, x^2)$.
Note that $I$ is Burch by \cite[Example 2.2 (2)]{DKT}. On the other hand, setting the ideal $L=(x^2)$ and $M=R/L$, we see $I\cap L=0=IL$ so that $\Tor_1^R(R/I, M)=0$. However $\pd_R(M)=\infty$, i.e, $M$ is not free.\pushQED{\qed} 
\qedhere
\popQED	
\end{ex}

Recall that it is proved in \cite[2.17]{CGTT} and \cite[1.5]{CT} that proper weakly $\m$-full ideals and powers of the maximal ideal, respectively, satisfy the torsion conclusion of Conjecture \ref{HWC}; see Corollary \ref{cor4}. It is clear that not every weakly $\m$-full ideal is a power of the maximal ideal; for example, a non-maximal prime ideal is such an example. It seems worth pointing out the converse, i.e., a power of the maximal ideal need not be weakly $\m$-full, in general.

\begin{ex}\label{notweaklymfull0} Let $R=k[\![t^4,t^5,t^{11}]\!]\cong k[\![x,y,z]\!]/(x^4-yz, y^3-xz,z^2-x^3y^2)$ be the numerical semigroup ring. One can check that $\m^2 \notni z \in (\m^3:_R\m)$.
Hence $\m^2$ is not a weakly $\m$-full ideal of $R$.\pushQED{\qed} 
\qedhere
\popQED	
\end{ex}

Next we give examples of non-weakly $\m$-full Burch ideals which are weakly $\m$-full with respect to a power of $\m$. More precisely, in Example \ref{e4},  
we construct an ideal which is Burch and weakly $\m$-full with respect to $\m$, but is not weakly $\m$-full. Similarly, Example \ref{notweaklymfull} yields a Burch ideal which is weakly $\m$-full with repsect to $\m^3$, but is not weakly $\m$-full (Recall that, if $I$ is a weakly $\m$-full ideal, then $I$ is weakly $\m$-full with respect to $\m^i$ for each $i\geq 0$; see Remark \ref{l5}.)

\begin{ex} \label{e4} Let $S=k[\![x,y,z, w]\!]$ and let $\a=(w^2, wy, wz^2, x^2, xy^2, xz, y^3, zy^2, yz^2, z^3)$. Letting $\n$ denote the maximal ideal of $S$, one can check that the following hold:
\begin{enumerate}[\rm(i)]
\item The given generating set for the ideal $\a$ is a minimal generating set.
\item $\n^3 \subseteq \a \subseteq \n^2$.
\item $\a \notni wx \in (\n \a:_R\n)$ so that $\a$ is not weakly $\n$-full.
\item $ \n^2 \subseteq (\a:_R \n) \subseteq (\n \a:_R\n^2)$.
\end{enumerate}

In fact we have that $(\n \a:_R \n^2) = \n^2$, i.e., $(\n \a:_R \n^2) \subseteq \n^2$ so that $\a$ is weakly $\n$-full with respect to $\n$. To see this suppose $(\n \a:_R \n^2) \nsubseteq \n^2$, and seek a contradiction. Hence there exists $l\in S$ such that $\n^2 \notni l\in \n$ and $l \n^2 \subseteq \n \a$. We can write $l=aw+bx+cy+dz+l'$ for some $a,b,c,d \in k$ and for some $l'\in \n^2$.
As, $l\n^2 \subseteq \n\a$, setting $l''=aw+bx+cy+dz$, we deduce that $l''\n^2 \subseteq \n \a$. In particular, $\n \a \ni l'' z^2=awz^2+cyz^2+dz^3$ (note $xz^2\in \n \a$). Since $wz^2, yz^2, z^3$ do not belong to the ideal $\n \a$ (they are part of a minimal generating set of $\a$), we conclude $a=c=d=0$ so that $l''=bx$. This implies that $\n \a \ni l'' y^2=bxy^2$. Since $xy^2$ does not belong to the ideal $\n \a$, we conclude that $b=0$. However, then $l=l'' \in \n^2$, which is a contradiction. Therefore we have $(\n \a:_R \n^2) = \n^2$.

Now we consider the ring $R=k[\![t^{256},t^{257}, t^{260}, t^{272}]\!]$, which is isomorphic to $S/J$, where $J=(z^4-x^3w, y^4-x^3z, x^{17}-w^{16})$. Note that $J\subseteq \n^4 \subseteq \n \a \subseteq \a$. Therefore, setting $I=\a/J$ and $\m=\n/J$, we have $I=\a/J \neq (\n \a:_S \n)/J=(\m I:_{R}\m)$ and  also $$(\n \a:_{S}\n^2)=(\a:_{S}\n) \Longrightarrow (\m I:_{R}\m^2)=(\n \a:_{S}\n^2)/J=(\a:_{S}\n)/J=(I:_{R}\m).$$
Consequently, $I$ is an ideal of $R$ which is is not weakly $\m$-full, but is weakly $\m$-full with respect to $\m$. Note also that $I$ is Burch; see Proposition \ref{prop4}. \pushQED{\qed} 
\qedhere
\popQED	
\end{ex}

\begin{ex}\label{notweaklymfull} Let $R=k[\![x,y]\!]$ and let $I=(x^5,x^3y,xy^3,y^5)$. Then $I$ is $\m$-primary since $\m^5\subsetneqq I\subsetneqq \m^4$. Moreover one can check that 
$(I:_R\m^3)=\m^2=(\m I:_R\m^4)$. Therefore, $I$ is weakly $\m$-full with respect to $\m^3$. However, since $I \notni x^2y^2\in (\m I:_R\m)$, we conclude that $I$ is not a weakly $\m$-full ideal. Furthermore, as $\ll(R/I)=5$, it follows from Proposition \ref{prop4} that $I$ is Burch.\pushQED{\qed} 
\qedhere
\popQED	
\end{ex}

If $I$ is an $\m$-primary ideal of a local ring $(R,\m)$ such that $(I:_R\m^s) \not= (\m I:_R\m^{s+1})$ for some integer $s>0$, then $I$ is not necessarily weakly $\m$-full and not necessarily weakly $\m$-full with respect to $\m^s$, in general. Furthermore, such an ideal $I$ may be Burch and may be weakly $\m$-full with respect to $\m^{s+1}$. We finish this section by addressing these properties with an example where $s=3$. More precisely, we give an example of a Burch ideal $I$ such that $I$ is not weakly $\m$-full but is weakly $\m$-full with respect to $\m^3$, $(I:_R\m^3) \neq (\m I:_R\m^4)$ and $(I:_{R}\m^4) = (\m I:_R\m^{5})$.

\begin{ex} \label{e6} Let $R=k[\![t^4,t^5,t^6]\!]$ and let $I=(t^{17}, t^{19}, t^{20})$. Then $\m^3=(t^{12}, t^{13}, t^{14}, t^{15})$, $\m^4=(t^{16}, t^{17}, t^{18}, t^{19})$ and $\m I=(t^{21}, t^{22}, t^{23}, t^{24})$. Note $I  \notni t^{18}=t^{6} \cdot t^{12} \in t^{6}  \m^3$ so that we have $t^6 \notin (I:_R\m^3)$. On the other hand, it follows $t^6 \m^4=(t^{22}, t^{23}, t^{24}, t^{25})\subseteq \m I$, i.e., $t^{6} \in (\m I:_R\m^4)$. Therefore we have that $(I:_R\m^3) \neq (\m I:_R\m^4)$. Hence, $I$ is not weakly $\m$-full with respect to $\m^3$. Furthermore Remark \ref{l5} shows $I$ is not a weakly $\m$-full ideal. 

One can observe that $\m^5\subsetneqq I \subsetneqq \m^4$. Hence $\m \subseteq (I:_{R}\m^4) \neq R$, i.e., $(I:_{R}\m^4) =\m$. Moreover, $t^{16} \in \m^4$ and $t^4 \in R$ so that $t^{20} \in \m^5$, but $t^{20}\notin \m I$, i.e., $\m^5 \nsubseteq \m I$. Thus it follows $\m \subseteq (\m I:_R\m^{5}) \neq R$ and so $(I:_{R}\m^4) =\m = (\m I:_R\m^{5})$. Consequently, $I$ is weakly $\m$-full with respect to $\m^4$. As $\ll(R/I)=5$, we conclude by Proposition \ref{c1} that $I$ is Burch. \pushQED{\qed} 
\qedhere
\popQED	
\end{ex}

\end{document}